\documentclass[11pt]{article}
\usepackage{amsthm,amssymb,amsmath}
\usepackage{epsfig}

\usepackage{pgf,tikz}
\usetikzlibrary{arrows}

\usepackage{hyperref}
\usepackage{color}

\newtheorem{prelem}{{\bf Theorem}}

 \newtheorem{theorem}{Theorem}

\newtheorem{lemma}[theorem]{Lemma}
\newtheorem{observation}[theorem]{Observation}
\newtheorem{proposition}[theorem]{Proposition}

\theoremstyle{definition}

\theoremstyle{remark}

\newcommand{\diam}{\operatorname{diam}}
\title{On the total and strong version for Roman dominating functions in graphs}
\author{
S. Nazari-Moghaddam$^{(1)}$, M. Soroudi$^{(1)}$, S.M. Sheikholeslami$^{(1)}$ and I.G. Yero$^{(2)}$\vspace{5mm}
\\
$^{(1)}$ Department of Mathematics, Azarbaijan Shahid Madani University\\
Tabriz, I.R. Iran\\
$^{(2)}$ Departmento de Matem\'aticas, Universidad de C\'adiz\\
Av Ram\'on Puyol s/n, 11202 Algeciras, Spain
\vspace{5mm}\\
{\tt \{s.nazari;m.soroudi;s.m.sheikholeslami\}@azaruniv.ac.ir,}\\ {\tt ismael.gonzalez@uca.es}
\vspace{3mm}\\}
\begin{document}
\maketitle

\begin{abstract}
Consider a finite and simple graph $G=(V,E)$ with maximum degree $\Delta$.
A strong Roman dominating function over the graph $G$ is understood as a map $f : V (G)\rightarrow \{0, 1,\ldots , \left\lceil
\frac{\Delta}{2}\right\rceil+ 1\}$ which carries out the condition stating that all the vertices $v$ labeled $f(v)=0$ are adjacent to at least one another vertex $u$ that satisfies $f(u)\geq 1+ \left\lceil \frac{1}{2}\vert N(u)\cap V_0\vert \right\rceil$, such that
$V_0=\{v \in V \mid f(v)=0 \}$ and the notation $N(u)$ stands for the open neighborhood of $u$.
The total version of one strong Roman dominating function includes the additional property concerning the not existence of vertices of degree zero in the subgraph of $G$, induced by the set of vertices labeled with a positive value.
The minimum possible value for the sum $\omega(f)=f(V)=\sum_{v\in V} f(v)$ (also called the weight of $f$), taken amongst all existent total strong Roman dominating functions $f$ of $G$, is called the total strong Roman domination number of $G$, denoted by $\gamma_{StR}^t(G)$. This total and strong version of the Roman domination number (for graphs) is introduced in this research, and the study of its mathematical properties is therefore initiated. For instance,  we establish upper bounds for such parameter, and relate it with several parameters related to vertex domination in graphs, from which we remark the standard domination number, the total version of the standard domination number and the (strong) Roman domination number. In addition, among other results, we show that for any tree $T$ of order $n(T)\ge 3$, with maximum degree $\Delta(T)$ and $s(T)$ support vertices,
$\gamma_{StR}^t(T)\ge \left\lceil \frac{n(T)+s(T)}{\Delta(T)}\right\rceil+1$.

\vspace{1.5mm}

\vspace{0.9mm}

\noindent{$\mathrm{Keywords}$}:  Total strong Roman dominating function; (total) Roman domination number; total dominating sets; Roman domination.\\
\end{abstract}
%.................................................................
\section{Introduction}

In all our exposition, we assume $G$ is a simple graph whose vertex set is $V(G)$ and its
edge set is $E(G)$ ($V$ and $E$ for more simplicity) such that  $G$ has no isolated vertex. The order of $G$ (or number of vertices of $G$) is given by $|V|$, and usually represented by $n=n(G)$. Given a vertex $v\in V(G)$, the {\em open neighborhood} of this vertex is the set $N_G(v)=N(v)=\{u\in V(G)\mid uv\in E(G)\}$. Moreover, the {\em closed neighborhood} of $v$ is the set $N_G[v]=N[v]=N(v)\cup \{v\}$. By using this, the {\em degree} of such vertex $v\in V$ in the graph $G$ can be written as $\deg_G(v)=d(v)=|N(v)|$. In concordance, the {\em maximum}
and {\em minimum degrees} of the graph $G$ are $\Delta=\Delta(G)$ and
$\delta=\delta(G)$, respectively. Now, for a fixed set of vertices $S$, the {\em open neighborhood} of such set $S\subseteq V$ is taken as $N_G(S)=N(S)=\cup_{v\in S}N(v)$, and the similarly, {\em closed neighborhood} of $S$ is $N_G[S]=N[S]=N(S)\cup S$. The {\em diameter} of $G$, which is denoted as $\diam(G)$, represents the maximum possible value among all the minimum distances between any pairs of vertices of the graph $G$. Now, concerning induced cycles in $G$, the {\em  girth} $g(G)$ of $G$ is taken as the length of one shortest cycle in $G$. We make the assumption that those graphs with no cycle (trees) have girth equal to $\infty$.
%The {\em $S$-external private neighborhood} of a vertex $v \in S$ is defined
%by $epn(v, S ) = \{w \in V (G) \setminus S \mid N[w] \cap S = \{v\}\}$.
%A {\em packing} in $G$ is a set of vertices that are pairwise at distance at least 3
%apart; that is, if $u$ and $v$ are distinct vertices that belong to a packing $S$, then
%$d(u, v) \geq 3$. Equivalently, a set $S$ of vertices of $G$ is a packing in $G$ if the closed
%neighborhoods of vertices in $S$ are pairwise disjoint.

A vertex of degree one in $G$ it is said to be a {\em leaf} of $G$. Also, a vertex being adjacent to a leaf is said to be a {\em
support vertex}. Now, for a given support vertex $v\in V(G)$, by $L_v$ we represent the set of all that leaves which are adjacent to $v$.
The special case of a tree obtained from two stars $K_{1,p}$ and $K_{1,q}, (p\ge q\ge 1)$ by adding an edge between their central vertices is called a {\em double star} $DS_{p,q}$. The following terminology for a rooted tree $T$ is used: $C(v)$ represents the set of all
children of a vertex $v\in V(T)$; $D(v)$ is taken as the set of descendants of $v$; and in addition we consider the set $D[v]=D(v)\cup \{v\}$. On the other hand, the \emph{depth} of the vertex $v$, written ${\rm depth}(v)$, is the largest distance between such $v$ and other vertex belonging to the set $D(v)$. Finally, by the subtree $T_v$ of $T$, we mean such {\em maximal subtree} induced by $v$ and the descendants $D(v)$ of $v$ (that is the set subtree induced by $D[v]$).

Domination theory is an classical and yet very popular research topic on graph theory. The number of open problems and lines of investigation concerning domination grows each day, as well as the number on researchers dedicating time to this topic. There is much theoretical knowledge, but also, much applications to practical problems can be found in the literature. To see some background on domination in graphs we suggest the books \cite{hhs,hhs-2}. One common research line is nowadays related to study different variants of the standard domination concept. Maybe the most common variations are the total domination (see \cite{He09,HeYe_book}), the independent domination (see \cite{GH}) and the Roman domination (see \cite{CDH04}). Each one of these three mentioned variants has itself its own variants (which are variants of the standard domination also). In this work, we make a contribution to a kind of combination of two of these variants, namely the total and the Roman domination.

A given set $S$ of vertices of the graph $G$ is taken as a \textit{dominating set} of $G$ if every vertex of $G$ not in $S$ has a neighbor in $S$, or equivalently, if $N[S]=V$.
The \textit{domination number} of $G$, from now on denoted by $\gamma(G)$, is then the possible minimum cardinality among all the existent
dominating sets of $G$. In connection with this, by a $\gamma(G)$-set, we mean a dominating set having the minimum possible cardinality in
$G$. A vertex set $D$ in $G$ is an \emph{efficient dominating set} for $G$ if for every vertex $v \in V(G)$,
there is exactly one $u \in D$ dominating $v$ (a vertex of $D$ dominates itself).
The set of vertices $S$ is a {\em total dominating set}, abbreviated TD-set, if it satisfies that
$N(S)=V$ (every vertex of $G$ has a neighbor in $S$, including also the vertices of $S$). The {\em total domination number}, which is usually denoted by $\gamma_t(G)$,  is understood as the minimum cardinality among all total dominating sets that exist in $G$.

Roman dominating functions in graphs were first formally defined by Cockayne \emph{et al.} in \cite{CDH04}, and this was partly motivated by a work of Ian Stewart \cite{s}. The idea comes from the ancient Roman Empire and a strategy of securely protect the empire from external attacks. In the last recent years, there has been an ``explosion'' of research works concerning Roman dominating functions, and by now, this topic is very well studied in its standard way. The literature on this topic has been detailed in several surveys
\cite{cjsv, cjsv1, cjsv2, cjsv3, cjsv4}. 
%(see for instance \cite{FaKaKhSh09,LiCh12a,LiCh12b,LiCh13,PaZe12,rr,GoRoJu13}).
However, there are still several ongoing works and open problems that are of high interest. Variations of the standard Roman domination are giving more insight into the classical problem, and new strategies of ``protecting a hypothetical Roman Empire'' are always of interest for the research community. As already mentioned, we are aimed now to continue contributing to the topic of Roman dominating functions in graphs.

A \emph{Roman dominating function} defined over the graph $G$, from now on RD-function for short,
is a map $f\colon V(G)\rightarrow \{0,1,2\}$ that satisfies the condition
that if $u$ is a vertex for which $f(u)=0$, then it must be adjacent to at least one other vertex $v$ such that $f(v)=2$. The \emph{weight} of the function $f$, written $\omega (f)$, is then $f(V(G))=\sum_{v\in V(G)}f(v)$. The \emph{Roman domination number}, which is commonly denoted by $\gamma _{R}(G)$, is the minimum weight among all RD-functions of the graph $G$. An RD-function with the minimum possible
weight $\gamma _{R}(G)$ in $G$ is called to be a $\gamma _{R}(G)$-\emph{function}.
For an RD-function $f$, we assume $V_{i}^{f}=\{v\in V(G)\colon f(v)=i\}$ for $i=0,1,2$. Observe that these three sets uniquely determine $f$, and so, it can be equivalently written $f=(V_{0}^{f},V_{1}^{f},V_{2}^{f})$. Also, note that $\omega(f)=|V_{1}^{f}|+2|V_{2}^{f}|$.

One of the most recent and interesting variations of Roman dominating functions in graphs, introduced by Liu and Chang in \cite{LiCh13} (although in a more general setting), is that one concerning the total Roman domination concept. That is, a \emph{total Roman dominating function} of a graph $G$ having no vertices of degree zero (or TRD-function for short), is a Roman dominating function $f$ on $G$, having an additional property which states that the subgraph of $G$ induced by the set of all vertices having a label with positive value under $f$, has minimum degree at least one. The \emph{total Roman domination number}, which is now on denoted by $\gamma _{tR}(G)$, is given as the minimum possible weight among every
TRD-functions on the graph $G$. A $\gamma _{tR}(G)$-function is taken as a TRD-function with weight $\gamma _{tR}(G)$. As already mentioned, the total Roman domination parameter was introduced by Liu and Chang \cite{LiCh13}, and studied for a more general albeit. Further on, specific studies on such parameters were developed for instance in the articles \cite{ah}-\cite{ah2-3}, \cite{ah3}-\cite{ah5}, \cite{ck,mky}. This leads to say that, nowadays, total Roman domination in graphs is relatively well studied.

The defensive strategy of Roman domination states that a vertex labeled with zero (unsecured place) must have at least a neighbor labeled with two (secured place). This means that if an unsecured position is attacked by one neighbor, then a secured (or stronger) neighbor could send one
of the two legions it possess, in order to defend this neighbor vertex from the performed attack. However, if there is secured place which has many neighbors labeled with zero, and a ``kind of simultaneous'' attack occurs, then this secured place cannot properly proceed. In order to deal with such situation, in \cite{S1} was introduced the idea of strong Roman dominating functions in graphs.

We now consider a map $f:V (G)\rightarrow \{0, 1,\ldots , \left\lceil \frac{\Delta}{2}\right\rceil+ 1\}$ that labels the vertices of $G$. Let $B_j=\{v\in V: f(v)=j\}$ for $j=0,1$ and let $B_2=V\setminus (B_0\cup B_1)=\{v\in V:f(v)\ge 2\}$
Then $f$ is a {\em strong Roman dominating function} on a graph $G$, abbreviated StRD-function, if any $v\in B_0$ has a neighbour $u$ for which $u\in B_2$ and such that $f(u)\geq 1+ \left\lceil \frac{1}{2}\vert N(u)\cap B_0\vert \right\rceil$.

Total Roman domination can be also seen as a kind of strategy of protection in which isolated elements are avoided. For the strong Roman domination, the ``weakness'' of protecting elements is dealt with. However, the isolated elements can still occur. It is then our goal to develop an protection strategy in which isolated elements are avoided as well as strong protection is considered. That is, a variant of Roman domination which we call total strong Roman domination.

A {\em total strong Roman dominating function}, abbreviated TStRD-function, represents a strong Roman dominating function satisfying also
that the set of all vertices having a label with positive value induce a graph with minimum degree at least one.
The minimum possible weight $\omega(f)=f(V)=\sum_{v\in V} f(v)$, among all the existent total strong Roman dominating functions $f$ of the graph $G$, is called the {\em total strong Roman domination number} of $G$ and is denoted by $\gamma_{StR}^t(G)$.
A TStRD-function $f$ is called a $\gamma^t_{StR}(G)$-function if $\omega(f)=\gamma^t_{StR}(G)$.

It is now our aim to initiate the study of this total and strong variant of the Roman domination number of graphs. We begin by establishing several tight upper bounds on the total strong Roman domination number of graphs concerning some graphs parameters and invariants.
Moreover, we relate this total and strong version of Roman domination to some other (standard or not) domination
parameters, like for example, the standard domination number, the total version of the domination number
and the strong Roman domination number.
In addition, we study the behavior of our parameter in question over the structure of a tree graph. We for instance prove that, for any given tree $T$ of order $n(T)\ge 3$, having maximum degree $\Delta(T)$ and $s(T)$ support vertices, it follows
$\gamma_{stR}^t(T)\ge \left\lceil \frac{n(T)+s(T)}{\Delta(T)}\right\rceil+1$.

We make use of the following results in this paper, some of which are straightforward to prove or to observe, and other ones are some known results.

\begin{observation}\label{1}
{\em For every graph $G$ with no isolated vertex, $\gamma_{tR}(G) \leq \gamma^t_{StR}(G) \leq \left\lceil \frac{\Delta+1}{2} \right\rceil\gamma_t(G)$.}
\end{observation}

\begin{observation}\label{o2}
{\em Let $G$ be a graph of order $n \geq 3$ with no isolated vertex. Then $$ 3 \leq \gamma^t_{StR}(G) \leq n.$$}
\end{observation}

\begin{observation}\label{ab}
{\em Let $G$ be a connected graph of order at least three and let $f=(B_0, B_1, B_2)$ be a $\gamma^t_{StR}(G)$-function.
Then,
\begin{enumerate}
  \item $|B_2| \leq  |B_0|$.
  \item If $x$ is a leaf and $y$ a support vertex in $G$, then $x \notin B_2$
and $y \notin B_0$.
\end{enumerate}
}
\end{observation}

\begin{prelem}\label{new1}
{\em (\cite{S1}) For any connected graph $G$ with $\Delta \leq 2$,
$\gamma_{StR}(G) = \gamma_{R}(G)$.}
\end{prelem}

\begin{observation}\label{o3}
{\em For any connected graph $G$ with $\Delta \leq 3$,
$\gamma^t_{StR}(G) = \gamma_{tR}(G)$.}
\end{observation}

\begin{prelem}\label{path}
{\em (\cite{C}) For any path $P_n$ and any cycle $C_n$, $\gamma_{R}(P_n)=\gamma_{R}(C_n)=\left\lceil \frac{2n}{3}\right\rceil$.}
\end{prelem}

\begin{prelem}\label{path2}
{\em (\cite{ah}) If $G$ is a nontrivial path or a cycle on $n$ vertices, then {\color{black}$\gamma_{tR}(G)=n$.}}
\end{prelem}

Let $\mathcal G$ be such family of graphs that can be obtained as follows. We begin with a 4-cycle $(v_1v_2v_3v_4)$. We next add $k_1+k_2\ge 1$ different paths $P_2$ (that is with vertex disjoint sets), and join $v_1$ to the end of $k_1$ paths of them, and also join $v_2$ to the end of
the remaining $k_2$ paths (it can possibly be, $k_1 = 0$ or $k_2 = 0$). Moreover, we consider the family $\mathcal H$ containing all the graphs that can be constructed from a double star by making one subdivision of each pendant edge of the double star, and also making a subdivision of the the non-pendant edge with $r \geq 0 $ vertices.

This family of graphs has been introduced in \cite{ah}.
\begin{prelem}\label{ah}
{\em (\cite{ah}) Let $G$ be a connected graph of order $n$. Then $\gamma_{tR}(G)=n$, if and only if one of the following holds.
\begin{enumerate}
\item[1.]
$G$ is a path or a cycle.
\item[2.]
$G$ is a corona, $cor(F $), of some graph $F$.
\item[3.]
 $G$ is a subdivided star.
\item[4.]
$G \in \mathcal{G} \cup \mathcal{H}$.
\end{enumerate}
}
\end{prelem}
\begin{prelem}\label{th4}
{\em (\cite{ah}) If $G$ is a graph with no isolated vertex, then
$\gamma_t(G)=\gamma_{tR}(G)$ if and only if $G$ is the disjoint union of copies of $K_2$.}
\end{prelem}

\begin{prelem}\label{th5}
{\em (\cite{ah}) Let $G$ be a connected graph of order $n \geq 3$. Then, $\gamma_{tR}(G)=\gamma_t (G)+1$ if and only if $\Delta(G)=n-1$.}
\end{prelem}

\begin{prelem}\label{th6}
{\em Let $G$ be a graph of minimum degree at least one. Then, (a): \cite{ore} $\gamma(G) \leq |V(G)|/2$ and,
(b): \cite{fjkr, px}  $\gamma(G) = |V(G)|/2$
 if and only if the components of $G$ are a cycle $C_4$ or the corona $H\circ K_1$ of any connected graph $H$.}
 \end{prelem}

%\begin{proposition}\label{t3}
%{\em For a connected graph $G$ with no isolated vertices, $\gamma^t_{StR}(G)=4$ if and only if either $G \in \{K_{1,4},K_{1,5},P_4,DS_{1,2},DS_{2,2}\}$}.
%\end{proposition}
%\begin{proof}
%Let $\gamma^t_{StR}(G) = 4$ and $f=(B_0,B_1,B_2)$ be a $\gamma^t_{StR}$-function on $G$.
%If $|B_2|=2$, then there exist two adjacent vertex $u, v \in V(G)$ such that $1+\left\lceil\frac{N(u)\cap B_0}{2} \right\rceil =2 $ and
%$1+\left\lceil\frac{N(v)\cap B_0}{2} \right\rceil =2 $. Hence  $G \in \{P_4,DS_{1,2},DS_{2,2}\}$. If $B_2=\{u\}, B_1=\{v\}$, then
%$u$ and $v$ are adjacent and $1+\left\lceil\frac{N(u)\cap B_0}{2} \right\rceil =3 $ and so $G \in \{K_{1,4},K_{1,5}\}$.
%If $|B_2|=1$, $|B_1|=2$, then let $B_2=\{u\}$ and $B_1=\{v,w\}$. Then $G \in \{DS_{1,2}, K_{1,4}\}$.
%\end{proof}
%***********************************************************************************
%*************************************
\section{Bounds on the total strong Roman domination number}

The main goal of this section concerns finding a few interesting closed bounds for the total strong Roman domination number of graphs, in which we relate it to other parameters or invariants of the graph. To this end, we need the following concepts. A \emph{matching} (also understood as an independent edge set) in a graph $G$ is formed by a set of edges having no vertices in common. The \emph{matching number} of the graph $G$ is the maximum cardinality among all possible matchings in $G$, which we denote by $\alpha'(G)$.

\begin{theorem}\label{S}
{\em Let  $G$ be a graph of order $n\ge 4$, maximum degree $\Delta$, and without isolated vertices and different from a star. Then,
$$\gamma^t_{StR}(G)\le n-\Delta+\alpha'(G)\left\lceil \frac{\Delta-1}{2}\right\rceil.$$}
\end{theorem}

\begin{proof}
We note that $\alpha'(G)\ge 2$ because $G$ is not a star. Let $v$ be a vertex of maximum degree $\Delta$ and let $X=V(G)\setminus N_{G}[v]$. If $X=\emptyset$, then clearly $v$ has degree $\Delta=n-1$ and also, it is satisfied $\gamma^t_{StR}(G)\leq\left\lceil \frac{\Delta-1}{2}\right\rceil+2$.  Hence,
$$\gamma^t_{StR}(G)\le 1+\left\lceil\frac{\Delta-1}{2}\right\rceil+1= n-\Delta+\left\lceil \frac{\Delta-1}{2}\right\rceil+1\le n-\Delta+\alpha'(G)\left\lceil \frac{\Delta-1}{2}\right\rceil,$$ which gives the desired bound.

Assume now that $X\neq \emptyset$ and that $S$ is the set consisting of all isolated vertices of the subgraph induced by $X$, now on denoted $G[X]$. If $S=\emptyset$, then let $u\in N(v)$ and define $f:V(G)\rightarrow \{0,1, \ldots, \left\lceil\frac{\Delta}{2}\right\rceil+1\}$ by $f(v)=\left\lceil \frac{\Delta-1}{2}\right\rceil+1$, $f(u)=1$, $f(x)=1$ for every $x\in X$, and $f(y)=0$ otherwise. Since $S$ is empty, we notice that $f$ is a TStRD-function of $G$ of weight $n-\Delta+\left\lceil \frac{\Delta-1}{2}\right\rceil+1$ yielding $$\gamma^t_{StR}(G)\le n-\Delta+\left\lceil \frac{\Delta-1}{2}\right\rceil+1\le n-\Delta+\alpha'(G)\left\lceil \frac{\Delta-1}{2}\right\rceil.$$

Let $S\neq \emptyset$. Since $\delta \ge 1$, every vertex $s\in S$ is adjacent to at least one vertex of $N(v)$. Let $S'$ be the smallest subset of $N(v)$ such that every vertex in $S$ is adjacent to a vertex of $S'$. By the choice of $S'$, each vertex $u'\in S'$ has a private neighbor $u\in S$ with respect to $S'$, and so $|S'|\le |S|$. Let $M=\{uu'\,:\,u'\in S'\mbox{ and } u\in S \mbox{ is a private neighbor of } u'\}$. Obviously, $M$ is a matching in $G$ yielding $|S'|\le \alpha'(G)$. We now consider four cases.

\smallskip
\noindent {\bf Case 1.} $S'=N(v)$ and $S=X$.\\ Hence, the function $f:V(G)\rightarrow \{0,1, \ldots, \left\lceil\frac{\Delta}{2}\right\rceil+1\}$  defined as $f(v)=1$, $f(x)=\left\lceil \frac{\Delta-1}{2}\right\rceil+1$ for every $x\in S'$ and $f(x)=0$ otherwise, is a TStRD-function of $G$ and so,
\begin{align}
\gamma^t_{StR}(G)&\le \left(\left\lceil \frac{\Delta-1}{2}\right\rceil+1\right)|S'|+1 = 1+|S'|+\left(\left\lceil \frac{\Delta-1}{2}\right\rceil\right)|S'|\nonumber \\
&\leq 1+|S|+\left(\left\lceil \frac{\Delta-1}{2}\right\rceil\right)|S'| \leq n-\Delta+\left(\left\lceil \frac{\Delta-1}{2}\right\rceil\right)|S'|\nonumber \\
&\leq n-\Delta+\alpha'(G)\left\lceil \frac{\Delta-1}{2}\right\rceil.\label{proof-1}
\end{align}
%\begin{equation}\label{e1}
%
%\end{equation}

\noindent {\bf Case 2.} $S'=N(v)$ and $S\subsetneqq X$.\\ Assume that $ww'\in E(G[X-S])$. Then $M\cup \{ww'\}$ is a matching of $G$ and so, $|S'|\le \alpha'(G)-1$. Define the function $f:V(G)\rightarrow \{0,1, \ldots, \left\lceil\frac{\Delta}{2}\right\rceil+1\}$ as $f(v)=1$, $f(x)=\left\lceil \frac{\Delta-1}{2}\right\rceil+1$ for any $x\in S'$, $f(x)=1$ for every $x\in X-S$, and $f(x)=0$ otherwise. Clearly, $f$ is a TStRD-function of $G$, and this implies that
\begin{align*}
\gamma^t_{StR}(G)& \leq \left(\left\lceil \frac{\Delta-1}{2}\right\rceil+1\right)|S'|+1+n-1-\Delta-|S|\\
&\leq n-\Delta+\left(\left\lceil \frac{\Delta-1}{2}\right\rceil\right)|S'|\\
&\leq n-\Delta+\left(\left\lceil \frac{\Delta-1}{2}\right\rceil\right)(\alpha'(G)-1)\\
&<n-\Delta+\alpha'(G)\left\lceil \frac{\Delta-1}{2}\right\rceil.
\end{align*}

\smallskip
\noindent {\bf Case 3.} $S'\subsetneqq N(v)$ and $S=X$.\\
Let $z\in N(v)-S'$. Hence, $M\cup \{vz\}$ is a matching of $G$ and so, $|S'|\le \alpha'(G)-1$.  We consider the function $f:V(G)\rightarrow \{0,1, \ldots, \left\lceil\frac{\Delta}{2}\right\rceil+1\}$ defined as $f(v)= 1+\left\lceil \frac{\Delta-|S'|}{2}\right\rceil$, $f(x)=\left\lceil \frac{\Delta-1}{2}\right\rceil+1$ for every $x\in S'$, and $f(x)=0$ otherwise. Note that $f$ is a TStRD-function of $G$ and so,
\begin{align}
 \gamma^t_{StR}(G) &\le 1+\left\lceil \frac{\Delta-|S'|}{2}\right\rceil+\left(\left\lceil \frac{\Delta-1}{2}\right\rceil+1\right)|S'|\nonumber \\
 &=1+|S'|+\left\lceil \frac{\Delta-|S'|}{2}\right\rceil+\left\lceil \frac{\Delta-1}{2}\right\rceil|S'|\nonumber\\
 & \leq 1+|S|+\left\lceil \frac{\Delta-|S'|}{2}\right\rceil+\left\lceil \frac{\Delta-1}{2}\right\rceil( \alpha'(G)-1)\nonumber \\
 &\leq n-\Delta+\left\lceil \frac{\Delta-1}{2}\right\rceil\alpha'(G).\label{proof-2}
\end{align}
\smallskip
\noindent {\bf Case 4.} $S'\subsetneqq N(v)$ and $S\subsetneqq X$.\\
Suppose that $z\in N(v)-S'$ and that $uu'\in E(G[X-S])$. Then clearly $M\cup \{vz,uu'\}$ is a matching of $G$ and so, we have $|S'|\le \alpha'(G)-2$. Define the function $f:V(G)\rightarrow \{0,1, \ldots, \left\lceil\frac{\Delta}{2}\right\rceil+1\}$ by $f(v)=\left\lceil \frac{\Delta-|S'|}{2}\right\rceil+1$, $f(x)=\left\lceil \frac{\Delta-1}{2}\right\rceil+1$ if $x\in S'$, $f(x)=1$ if $x\in X-S$ and $f(x)=0$ otherwise. It is easy to see that $f$ is a TStRD-function of $G$. Thus,
\begin{align*}
 \gamma^t_{StR}(G) &\le n-1-\Delta-|S|+1+\left\lceil \frac{\Delta-|S'|}{2}\right\rceil+\left(\left\lceil \frac{\Delta-1}{2}\right\rceil+1\right)|S'|\\
 &\leq n-\Delta+\left\lceil \frac{\Delta-|S'|}{2}\right\rceil+\left(\left\lceil \frac{\Delta-1}{2}\right\rceil\right)|S'|\\
&\leq n-\Delta+\left\lceil \frac{\Delta-|S'|}{2}\right\rceil+\left(\left\lceil \frac{\Delta-1}{2}\right\rceil\right)(\alpha'(G)-2)\\
 %&\leq n-\Delta+(\left\lceil \frac{\Delta-1}{2}\right\rceil)(\alpha'(G)-1)\\
 &<n-\Delta+\alpha'(G)\left\lceil \frac{\Delta-1}{2}\right\rceil,
\end{align*}
and the proof is complete.
\end{proof}

We next characterize the graphs attaining the aforementioned bound among those graphs having girth at least four. To this end, we need the set of graphs appearing in Fig. \ref{fig1}. Also, for a given graph $G$, by $S(G)$ we mean a graph obtained from $G$ by subdividing all its edges.

%\vspace{-0.1 cm}
\begin{figure}[ht]
\def\emline#1#2#3#4#5#6{%
\put(#1,#2){\special{em:moveto}}
\put(#4,#5){\special{em:lineto}}}
\def\newpic#1{}
\begin{center} \unitlength 0.9mm
\begin{tikzpicture}[line cap=round,line join=round,>=triangle 45,x=1.0cm,y=1.0cm]
\clip(2.75,-0.34) rectangle (11.33,3.74);
\draw (4,2.5)-- (3,3.5);
\draw (5,3.5)-- (4,3.5);
\draw (3,3.5)-- (3,2.5);
\draw (5,3.5)-- (5,2.5);
\draw (3,3.5)-- (4,3.5);
\draw (3.95,2.39) node[anchor=north west] {$F_1$};
\draw (7,2.5)-- (6,3.5);
\draw (8,3.5)-- (7,3.5);
\draw (6,3.5)-- (6,2.5);
\draw (8,3.5)-- (8,2.5);
\draw (6,3.5)-- (7,3.5);
\draw (6.95,2.37) node[anchor=north west] {$F_2$};
\draw (7,2.5)-- (8,3.5);
\draw (9,2.5)-- (9,3.5);
\draw (11,3.5)-- (10,3.5);
\draw (9,3.5)-- (10,2.5);
\draw (11,3.5)-- (11,2.5);
\draw (9,3.5)-- (10,3.5);
\draw (9.95,2.37) node[anchor=north west] {$F_3$};
\draw (9,2.5)-- (11,3.5);
\draw (11,3.5)-- (10,2.5);
\draw (7.5,1)-- (8.5,1.5);
\draw (8.5,1.5)-- (9.5,1);
\draw (8.5,1)-- (8.5,0.5);
\draw (8.42,0.25) node[anchor=north west] {$F_5$};
\draw (8.5,1.5)-- (8.5,1);
\draw (8.5,0.5)-- (7.5,1);
\draw (8.5,0.5)-- (9.5,1);
\draw (5.5,0.5)-- (5.5,1.5);
\draw (5.5,1.5)-- (4.5,0.5);
\draw (6.5,1.5)-- (6.5,0.5);
\draw (5.9,0.26) node[anchor=north west] {$F_4$};
\draw (5.5,1.5)-- (6.5,1.5);
\draw (5.5,0.5)-- (6.5,0.5);
\begin{scriptsize}
\fill [color=black] (3,3.5) circle (1.5pt);
\fill [color=black] (4,3.5) circle (1.5pt);
\fill [color=black] (4,2.5) circle (1.5pt);
\fill [color=black] (5,3.5) circle (1.5pt);
\fill [color=black] (3,2.5) circle (1.5pt);
\fill [color=black] (5,2.5) circle (1.5pt);
\fill [color=black] (6,3.5) circle (1.5pt);
\fill [color=black] (7,3.5) circle (1.5pt);
\fill [color=black] (7,2.5) circle (1.5pt);
\fill [color=black] (8,3.5) circle (1.5pt);
\fill [color=black] (6,2.5) circle (1.5pt);
\fill [color=black] (8,2.5) circle (1.5pt);
\fill [color=black] (9,3.5) circle (1.5pt);
\fill [color=black] (10,3.5) circle (1.5pt);
\fill [color=black] (9,2.5) circle (1.5pt);
\fill [color=black] (11,3.5) circle (1.5pt);
\fill [color=black] (10,2.5) circle (1.5pt);
\fill [color=black] (11,2.5) circle (1.5pt);
\fill [color=black] (8.5,1.5) circle (1.5pt);
\fill [color=black] (7.5,1) circle (1.5pt);
\fill [color=black] (8.5,1) circle (1.5pt);
\fill [color=black] (9.5,1) circle (1.5pt);
\fill [color=black] (8.5,0.5) circle (1.5pt);
\fill [color=black] (5.5,1.5) circle (1.5pt);
\fill [color=black] (5.5,0.5) circle (1.5pt);
\fill [color=black] (6.5,1.5) circle (1.5pt);
\fill [color=black] (4.5,0.5) circle (1.5pt);
\fill [color=black] (6.5,0.5) circle (1.5pt);
\end{scriptsize}
\end{tikzpicture}
\end{center}
\vspace{-0.7 cm} \caption{The graphs $F_1,F_2,F_3,F_4$ and $F_5$} \label{fig1}
\end{figure}
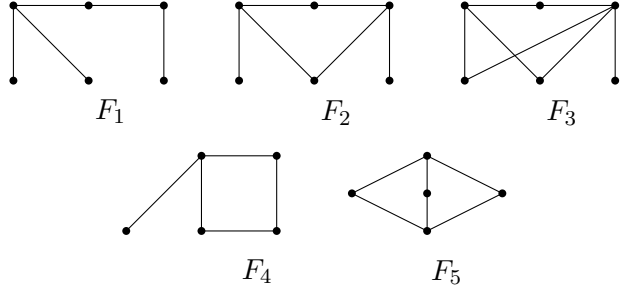

\begin{theorem}\label{girth}
{\em Let  $G$ be a connected graph of order $n$, maximum degree $\Delta$, and such that $g(G)\ge 4$. Then,
$\gamma^t_{StR}(G)= n-\Delta+\alpha'(G)\left\lceil \frac{\Delta-1}{2}\right\rceil$ if and only if $G$ is one of the graphs in the set $\{P_4,P_5,C_4,C_5,DS_{1,2},S(K_{1,3}),{\color{black}F_1,F_2,F_3,F_4,F_5}\}$.
}
\end{theorem}

\begin{proof}
%One side is clear.
Assume that $\gamma^t_{StR}(G)= n-\Delta+\alpha'(G)\left\lceil \frac{\Delta-1}{2}\right\rceil$. By Observation \ref{o2}, we have $n-\Delta+\alpha'(G)\left\lceil \frac{\Delta-1}{2}\right\rceil\le n$. First observe that if $\Delta=2$ or $\Delta\ge 4$, then $\alpha'(G)\le 2$. Also, if $\Delta=3$, then $\alpha'(G)\le 3$. If $\alpha'(G)=1$, then $G$ is a star and we get a contradiction.

Hence, we may assume that $\alpha'(G)=2$ or $\alpha'(G)=\Delta=3$. Let $v$ be a vertex of maximum degree $\Delta$, and let $X$, $S$ and $S'$ denote the sets previously defined in the proof of Theorem \ref{S}. Since $g(G)\ge 4$, we have $X\neq \emptyset$. First let $S=\emptyset$. Then the function  $f:V(G)\rightarrow \{0,1, \ldots, \left\lceil\frac{\Delta}{2}\right\rceil+1\}$ {\color{black}defined} by $f(v)=\left\lceil \frac{\Delta-1}{2}\right\rceil+1, f(u)=1, f(x)=1$ for $x\in X$ and $f(x)=0$ otherwise, is a TStRD-function of $G$ of weight $n-\Delta+\left\lceil \frac{\Delta-1}{2}\right\rceil+1$. Then
$n-\Delta+\alpha'(G)\left\lceil \frac{\Delta-1}{2}\right\rceil\le n-\Delta+\left\lceil \frac{\Delta-1}{2}\right\rceil+1$.
If $\alpha'(G)=\Delta=3$, then
$$n=n-3+3\left\lceil \frac{3-1}{2}\right\rceil\le n-3+\left\lceil \frac{3-1}{2}\right\rceil+1=n-1$$
which is a contradiction. Assume that $\alpha'(G)=2$.
It then follows that $n-\Delta+2\left\lceil \frac{\Delta-1}{2}\right\rceil\le n-\Delta+\left\lceil \frac{\Delta-1}{2}\right\rceil+1$, yielding $\Delta=2$ or $\Delta=3$. If $\Delta=2$, then, from Observation \ref{o3} and Theorem \ref{path2}, we deduce that $G \in \{P_5,C_5\}$.
Let $\Delta=3$. We deduce from $\alpha'(G)=2$ and $S=\emptyset$, that $G[X]$ is a star $K_{1,t}$ whose central vertex, say $z$, is adjacent to a neighbor of $v$, say $u$. Since $\Delta=3$, we have $t\in \{1,2\}$. If $t=2$, then the function $g:V(G)\rightarrow \{0,1, \ldots, \left\lceil\frac{\Delta}{2}\right\rceil+1\}$ defined by $g(v)=g(z)=2, g(u)=1$ and $g(y)=0$ otherwise, is a TStRD-function of $G$ of weight less than $n-\Delta(G)+\alpha'(G)\left\lceil \frac{\Delta-1}{2}\right\rceil=6$, which is a contradiction. Thus $t=1$. If $g(G)=\infty$, then we observe that $G$ is a graph obtained from double star $DS_{1,2}$ by subdividing its central edge once, and so $G=F_1$. Let $g(G)<\infty$. Since $\Delta(G)=3$ and $\alpha'(G)=2$, we have $g(G)=4$. Hence, $z$ must be adjacent to some other neighbors of $v$. It is easy to see that $G\in \{F_2,F_3\}$ in this case.

Now suppose that $S\neq \emptyset$. By the proof of Theorem \ref{S}, we only consider the following cases.

\smallskip
\noindent{\bf Case 1.} $S'=N(v)$ and $S=X$.\\
By \eqref{proof-1}, we deduce that $\alpha'(G)=|S'|=|S|=\Delta$. Since
$ n-\Delta+\Delta\left\lceil \frac{\Delta-1}{2}\right\rceil \leq n$, we have $\Delta \leq 3$. If $\Delta=2$, then clearly $G \in \{P_5,C_5\}$.
Let $\Delta=3$. Since $G$ is triangle-free, $N(v)$ is independent and since each vertex in $S'$
has a private neighborhood in $S$ with
respect to $S'$, we conclude that each vertex in $S$ is of degree one. Thus $G$ is obtained from $K_{1,3}$ by subdividing its edges, \emph{i.e.}, $G=S(K_{1,3})$.

\smallskip
\noindent{\bf Case 2.} $S'\subsetneqq N(v)$ and $S=X$.\\
In this case all inequalities occurring in \eqref{proof-2} must be equalities, and so $\alpha'(G)-1=|S'|=|S|=1$. Let $S'=\{u\}$.
Hence, the function $g:V(G)\rightarrow \{0,1, \ldots, \left\lceil\frac{\Delta}{2}\right\rceil+1\}$ defined by $g(v)=\left\lceil \frac{\Delta-1}{2}\right\rceil+1$, $g(u)=2$, and $g(y)=0$ otherwise, is a TStRD-function of $G$ of weight $\left\lceil \frac{\Delta-1}{2}\right\rceil+3$. We must have $\left\lceil\frac{\Delta-1}{2}\right\rceil+3\ge n-\Delta+2\left\lceil \frac{\Delta-1}{2}\right\rceil $, and so,
$n-\Delta+\left\lceil \frac{\Delta-1}{2}\right\rceil \le3$. Since $n=\Delta+2$, we obtain $\left\lceil \frac{\Delta-1}{2}\right\rceil\le1$, which leads to $\Delta \leq 3$. If $\Delta(G)=2$, then clearly $G\in \{P_4,C_4\}$. If $\Delta(G)=3$, then we note that $G\in \{F_4,F_5,DS_{1,2}\}$. This completes the proof.
\end{proof}

We now continue with some other bounds in which we also involve some other invariants of the graph, like the minimum degree, the diameter, and the girth.

\begin{proposition}
Let $G$ be a connected graph of order $n\ge 2$ and minimum degree $\delta$. Then,
$$\gamma^t_{StR}(G) \leq n-\left\lfloor \frac{\delta-1}{2}\right\rfloor.$$
\end{proposition}

\begin{proof}
Let $v \in V(G)$ be a vertex of minimum degree $\delta$ and let $u\in N(v)$. Define the function $f: V(G) \to \{1,2, \ldots, \left\lceil \frac{\Delta}{2}\right\rceil+1 \}$ such that $f(v)=\left\lceil \frac{\delta-1}{2}\right\rceil+1$, $f(u)=1$, $f(x)=0$ for any $x \in N(v)-\{u\}$, and $f(y)=1$ for every $y\notin N[v]$. Note that $f$ is a total strong Roman dominating function of $G$. Namely, if there is a vertex $w\notin N[v]$ that has no neighbor outside of $N[v]$, then $w$ must be adjacent to every neighbor of $v$, since otherwise $w$ is a vertex of degree smaller than $v$, which is not possible. Thus,
$\gamma_{StR}^t(G) \leq \left\lceil \frac{\delta-1}{2}\right\rceil+2+n-1-\delta=  n-\left\lfloor \frac{\delta-1}{2}\right\rfloor$, as desired.
\end{proof}

\begin{proposition}
{\em  Let $G$ be a graph  of order $n$  with $\diam(G)=2$, minimum degree $\delta$, and maximum degree $\Delta$. Then,
$$\gamma^t_{StR}(G)\leq \delta\left(1+\left\lceil\frac{\Delta-1}{2}\right\rceil\right)+1.$$}
\end{proposition}

\begin{proof}
Suppose $v \in V(G)$ is a vertex of minimum degree $\delta$. Clearly, $N(x)$ dominates all vertices of $G$ since $G$ has diameter two.
Define the function $f:V(G)\rightarrow \{0,1,\ldots,1+\left\lceil\frac{\Delta}{2}\right\rceil\}$ such that
$f(v)=1, f(x)=1+\left\lceil\frac{\Delta-1}{2}\right\rceil$ for every $x\in N(v)$ and
$f(x)=0$ otherwise. It is clear that $f$ is a total strong Roman dominating
function yielding $\gamma_{StR}^t(G)\leq \delta(1+\left\lceil\frac{\Delta-1}{2}\right\rceil)+1.$
\end{proof}
{\color{black}Next we establish upper bounds in terms of the order, diameter and girth of the graph.}
\begin{proposition}
{\em Let $G$ be a connected graph of order $n$ and minimum degree $\delta \ge 3$. Then,
$$\gamma_{StR}^t(G)\le n-\left\lfloor \frac{{\rm diam(G)}+1}{3}\right\rfloor.$$}
\end{proposition}

\begin{proof}
Assume $P=v_1v_2\ldots v_{{\rm diam}(G)+1}$ is a diametral path in $G$ and let $f$ be a $\gamma_{StR}(P)$-function. By  Theorems \ref{new1} and \ref{path}, we have $\omega (f)=\left\lceil \frac{2{\rm diam}(G)+2}{3}\right\rceil$. Hence, the function $g:V(G)\rightarrow \{0,1,\ldots,\left\lceil\frac{\Delta}{2}\right\rceil+1\}$ defined as $g(u)=f(u)$ for every $u\in V(P)$ and $g(u)=1$ for every $u\in V(G)\setminus V(P)$, is a TStRD-function of $G$. Therefore,
$$\gamma^t_{StR}(G)\le (n-{\rm diam}(G)-1)+\left\lceil \frac{2{\rm diam}(G)+2}{3}\right\rceil=n-\left\lfloor \frac{{\rm diam(G)}+1}{3}\right\rfloor,$$
which is our desired bound.
\end{proof}

\begin{proposition}
{\em Let $G$ be a connected graph of order $n$ with $g(G)\geq  4$ and $\delta\ge 3$. Then
$$\gamma_{StR}^t(G)\le n-\left\lfloor \frac{g(G)}{3}\right\rfloor.$$}
\end{proposition}

\begin{proof}
Let $C$ be a cycle of $G$ with $g(G)$ edges. For any $\gamma_{StR}(C)$-function, the function $g:V(G)\rightarrow \{0,1, \ldots, \left\lceil\frac{\Delta}{2}\right\rceil+1\}$ defined by $g(u)=f(u)$ for $u\in V(C)$ and $g(u)=1$ for $u\in V(G)\setminus V(C)$, is a TStRD-function of $G$. By Theorem \ref{new1} and \ref{path}, we have
$$\gamma^t_{StR}(G)\le \omega (f)+(n-g(G))=n-\left\lfloor \frac{g(G)}{3}\right\rfloor,$$
and the proof is complete.
\end{proof}

We conclude this section by characterizing all the graphs attaining the largest possible value in the total strong Roman domination number, and giving some Norhauss-Gaddum result for such parameter.

\begin{theorem}\label{thn}
{\em  Let $G$ be a connected graph of order $n$. Then $\gamma^t_{StR}(G)=n$, if and only if one of the next items is satisfied.
\begin{enumerate}
\item[(i)]
$G$ is a cycle or a path.
\item[(ii)]
$G$ is a thecorona, $cor(F $), of some other graph $F$.
\item[(iii)]
 $G$ is a subdivided star graph.
\item[(iv)]
$G \in \mathcal{G} \cup \mathcal{H}$, with $\mathcal{G}, \mathcal{H}$ as defined in Introduction.
\end{enumerate}}
\end{theorem}

\begin{proof}
Suppose that the graph $G$ satisfies at least one of the (four) conditions given in the statement of the theorem.
Hence, by using Observations \ref{1} and \ref{o2} and Theorem \ref{ah}, we deduce that $\gamma^t_{StR}(G)=n$.

Conversely, let $\gamma^t_{StR}(G)=n$. We claim that$\gamma^t_{StR}(G)=\gamma_{tR}(G)$.
Let $f=(B_0,B_1,B_2)$ be a $\gamma^t_{StR}(G)$-function. If there is one vertex, say $v \in B_2$, for which $f(v)\geq x$ where $x \geq 3$, then
$|N(v) \cap B_0|\geq 2x-3 \geq x $. Thus, $\gamma^t_{StR}(G) < n$, which is a contradiction. Consequently, every vertex in $B_2$ has value 2. Therefore, $\gamma^t_{StR}(G)=\gamma_{tR}(G)=n$ and by Theorem \ref{ah}, the graph $G$ satisfies one of the conditions appearing in
the statement of our result.
\end{proof}

\begin{proposition}\label{ng}
{\color{black}Let $G$ and $\overline{G}$ be connected graphs of order $n \geq 4$.}
Then  $8 \leq \gamma^t_{StR}(G) + \gamma^t_{StR}(\overline{G}) \leq 2n$.  Moreover,
%\begin{enumerate}
 $\gamma^t_{StR}(G) +\gamma^t_{StR}(\overline{G}) = 2n$ if and only if $G = P_4$.
%%\item[(ii)] $\gamma^t_{StR}(G) + \gamma^t_{StR}(\overline{G}) = 2n - 1$ if and only if $G =P_5$.
%\item[(ii)] $\gamma^t_{StR}(G) + \gamma^t_{StR}(\overline{G}) = 8$ if and only if
%$G = P_4$.
%\end{enumerate}
\end{proposition}

\begin{proof}
{\color{black}From the condition,
%both $G$ and   $\overline{G}$ have no isolated vertices.
%Hence
$\gamma^t_{StR}(G)$ and $\gamma^t_{StR}(\overline{G})$ exist.}
 Let without loss of generality,  $\gamma^t_{StR}(G) \geq \gamma^t_{StR}(\overline{G})$.
First we will prove the right side inequality.
 By Observation \ref{o2}, $\gamma^t_{StR}(G) + \gamma^t_{StR}(\overline{G}) \leq 2n$
 with equality if and only if $\gamma^t_{StR}(G) = \gamma^t_{StR}(\overline{G})=n$.
 Now by Theorem \ref{thn}, $G$ and $\overline{G}$ are $P_4$.
%(ii)  The equality $\gamma^t_{StR}(G) + \gamma^t_{StR}(\overline{G}) = 2n - 1 $ implies
%$\gamma^t_{StR}(G) = n$ and $\gamma^t_{StR}(\overline{G}) = n - 1$.
%If $x$ and $y$ are nonadjacent leaves in $G$, then
%the function $f: V(G) \to \{1,2, \ldots, \left\lceil \frac{\Delta}{2}\right\rceil+1 \}$
%defined by  $f(x)=f()$
%$f = (V(G) - \{x,y\}; \emptyset; \emptyset; \{x,y\})$
%is TStRDF on  $\overline{G}$  which leads to $n-1 \leq 4$. Thus $n \leq 5$.
%By Theorem \ref{thn},
%$G =P_4$ or $P_5$.	Since $\gamma^t_{StR}(\overline{G}) = n - 1$,
%$G = P_5$. 		
		
Now we will prove the left side inequality. Without loss of generality assume
$\gamma^t_{StR}(G)  \leq \gamma^t_{StR}(\overline{G})$. By Observation \ref{o2},
$\gamma^t_{StR}(G)  \geq 3$. {\color{black}If  $\gamma^t_{StR}(G)  = 3$, then clearly
$3=\gamma^t_{StR}(G)=\gamma_{tR}(G)=\gamma_t(G)+1$ and we conclude from
Theorem \ref{th5} that $G$ has order 4 and $\Delta=3$, but then $\overline{G}$
has isolated vertices, which is a contradiction.} Therefore, $\gamma^t_{StR}(G) \geq 4$ and so $\gamma^t_{StR}(\overline{G}) \geq 4$, which means $\gamma^t_{StR}(G) + \gamma^t_{StR}(\overline{G}) \geq 8$.
%If the equality holds, then Proposition \ref{t3} implies that $ G$ and $\overline{G} $ are $P_4$.
\end{proof}

\section{Total strong Roman domination versus  strong Roman domination, total domination  and domination}

We next present a bound for the total strong Roman domination number which is related to the strong Roman domination parameter.

\begin{theorem}
{\em
If $G$ is a graph of order $n\geq 4$ and minimum degree at least one, then $$\gamma^t_{StR}(G) \leq 2(\gamma_{StR}(G)-1).$$
This bound is sharp for $G\in \{P_4,C_4,P_6,C_6\}$.}
\end{theorem}

\begin{proof}
Let $f = (B^f_0, B^f_1, B^f_2)$ be a $\gamma_{StR}(G)$-function such that $|B^f_2|$ is
maximum. Suppose firstly that $|B^f_1| \neq 0$.
Let $B^f_{12}$ be the set formed by those vertices belonging to $B^f_1$ for which there is a neighbor in $B^f_2$.
%and let $V_{11}= B_1-B_{12}$.
Consider $|B^f_{12}| \neq 0$, and let $u \in B^f_{12}$ and $v \in N(u) \cap B^f_2$.
If $N(v) \cap B^f_0 $ is odd, then the function $g:V(G) \to \{0,1, \ldots, \left\lceil \frac{\Delta}{2}\right\rceil+1 \}$ defined by
$g(u)=0$ and $g(x)=f(x)$ for $x \in V(G)-\{u\}$ is an StRD-function
of weight less than $f(V(G))$, a contradiction.
Hence, $|N(v) \cap B^f_0|$ is even, for every vertex $v\in N(B^f_{12}) \cap B^f_2 $.
Let $B^f_{11}$ be the subset of $B^f_1$ such that $N(B^f_{11}) \subseteq B^f_0$.
%We claim that every vertex in $B_{12}$ is not adjacent to vertex in $B_{11}$.
%Let $v \in B_{12}$ and $u \in B_{11}$. Then
%the function $g:V(G) \to \{0,1, \ldots, \left\lceil \frac{\Delta+1}{2}\right\rceil\}$ defined by
%$g(v)=2, g(u)=0$ and $g(x)=f(x)$ for $x \in V(G)-\{u,v\}$ is an StRD-function
%of $G$ which is a contradiction with our choice of $f$.
%Vow, we show that $B_{11}$
%is an independent set in $G$. Suppose, to the contrary,
%that $u$ and $v$ are adjacent vertices in $B_{11}$.
%Then the function $g:V(G) \to \{0,1, \ldots, \left\lceil \frac{\Delta+1}{2}\right\rceil\}$
%by $g(u)=2, g(v)=0$ and  $g(x)=f(x)$ for $x \in V(G)-\{u,v\}$ is an StRD-function
%of $G$ which is a contradiction with our choice of $f$.
%Hence,
%$B_{11}$ is an independent set in $G$.
% We show next that $B_{11}$ is a packing in $G$.
% Suppose,
%to the contrary, that $u,v \in B_{11}$ such that $d(u,v)=2$ and $w \in N(u) \cap N(v)$.
%Since
%$B_2$
%is an independent set and
%neither $u$ nor $v$ are dominated by $B_2$ and there not exists a edge between $B_{11}$ and $B_{12}$, we have
%$w \in V_0$.
%Then the function $g:V(G) \to \{0,1, \ldots, \left\lceil \frac{\Delta+1}{2}\right\rceil\}$
%by $g(w)=2, g(u)=g(v)=0$ and  $g(x)=f(x)$ for $x \in V(G)-\{u,v,w\}$ is an StRD-function
%of $G$ which is a contradiction with our choice of $f$.
%Therefore,
%$B_{11}$
%is a packing in $G$.
Note that  all neighbors of a vertex in $B^f_{11}$
belong to $B^f_0$ and every vertex in $B^f_1-B^f_{11}$ has a neighbor with positive label.
For every vertex $w \in B^f_{11}$, we chose any of its neighbors neighbor, say $w'$, and we let
$W=\bigcup_{w \in B^f_{11}}{w'}$. Notice that $|W|\leq |B^f_{11}| \leq |B^f_1|$, and that $W \subseteq B^f_0$.
Since every vertex in $W \cup B_{12}^f$ is adjacent to at least
one vertex from the set $B^f_2$, we deduce $|B^f_2| \geq 1$ and $\sum_{v \in B^f_2} {\color{black}\left\lceil \frac{|N(v) \cap B^f_0|}{2}\right\rceil} \geq 1$.
If $G[B^f_1 \cup B^f_2]$ has no isolated vertex, then
$f$ is a $\gamma_{StR}^t$-function of $G$, which implies that $\gamma_{StR}^t(G)=\gamma_{StR}(G)< 2(\gamma_{StR}(G)-1)$.

Assume now that $G[B^f_1 \cup B^f_2]$ has  at least one isolated vertex.
Consider the function $h=(B_0^h,B_1^h,B_2^h)=(V_0-W,B_1 \cup W,B_2)$.
If $G[B_1^h \cup B^h_2]$ has no isolated vertex, then $h$ is a TStRD-function on $G$ such that
\begin{align*}
\gamma^t_{StR}(G) &\leq h(V(G)) \leq 2|B^h_1|+|B^h_2|+\sum_{y \in B^h_2} {\color{black}\left\lceil \frac{|N(y) \cap B^h_0|}{2}\right\rceil}\\
&=2\gamma_{StR}(G)-|B^h_2|-\sum_{y \in B^h_2} {\color{black}\left\lceil \frac{|N(y) \cap B^h_0|}{2}\right\rceil}\leq 2\gamma_{StR}(G)-2\\
 &=2(\gamma_{StR}(G)-1).
\end{align*}
If the subgraph induced by $G[B^h_1 \cup B^h_2]$ has a vertex of degree zero, then we consider $U$ as the set of
such vertices of degree zero in $G[B^h_1 \cup B^h_2]$.
Since the induced subgraph $G[B^h_1] $ has no vertex of degree zero, we are able to check that $U \subseteq B^h_2$. As we noticed before, every vertex from the set $W$ has at least one neighbor vertex in the set $B_2$ This implies that $U \subset B^h_2$. Now, for any vertex $u \in U$ , we chose one of its neighbors, say $u'$, and we make $M=\bigcup_{u \in U}\{u'\}$. Note that $|M| \leq |U|< |B^h_2|$  and $M \subseteq B^h_0$.
Then the function $h'=(B^{h'}_0,B^{h'}_1,B^{h'}_2)=(V^h_0-M,B^h_1 \cup M,B_2)$ is a TStRD-function of $G$, implying that
\begin{align*}
\gamma^t_{StR}(G) &\leq h'(V(G)) \leq 2|B^{h'}_1|+2|B^{h'}_2|+{\color{black}\sum_{z \in B^{h'}_2} \left\lceil \frac{|N(z) \cap B^{h'}_0|}{2}\right\rceil}-1\\
&=2\gamma_{StR}(G)-{\color{black}\sum_{z \in B^{h'}_2} \left\lceil \frac{|N(z) \cap B^{h'}_0|}{2}\right\rceil}-1\leq 2\gamma_{StR}(G)-2.\\
\end{align*}

We now consider $|B^f_1|=0$, and assume $U$ is formed by the set of vertices of degree zero in the induced subgraph $G[B^f_2]$.
If $|U|=0$, then we obtain $\gamma^t_{StR}(G)=\gamma_{StR}(G)$.
Hence, we may consider that $|U|\neq 0$.
For every vertex $u \in U $, we take any neighbor of it, say $u'$, and we do
$U'=\bigcup_{u\in U}\{u'\}$. Notice that we have $|U'|\leq |U|\leq |B_2^f| $.
Then the function $h''=(B_0^{h''},B_1^{h''},B_2^{h''})=(B^f_0-U',U',B^f_2)$ is a TStRD-function of $G$, implying that
\begin{align*}
\gamma^t_{StR}(G) &\leq h''(V(G)) \leq 2|B^{h''}_2|+{\color{black}\sum_{w \in B^{h''}_2} \left\lceil \frac{|N(w) \cap B^{h''}_0|}{2}\right\rceil}\\
&=2\gamma^t_{StR}(G)-{\color{black}\sum_{w \in B^{h''}_2} \left\lceil \frac{|N(w) \cap B^{h''}_0|}{2}\right\rceil|}\leq 2\gamma_{StR}(G)-1\\
\end{align*}
If $\gamma^t_{StR}(G)=2\gamma_{StR}(G)-1$, then $\gamma_{StR}(G) = 2$, implying that $n \leq 3$, which is a contradiction. Therefore
$\gamma^t_{StR}(G)\leq 2(\gamma_{StR}(G)-1)$ and this completes the proof.
\end{proof}

Since the vertices labeled with positive numbers in any TStrRD-function of a graph $G$ form a total dominating set of $G$, it clearly happens that $\gamma^t_{StR}(G)\ge \gamma_t(G)$. We are next interested into characterizing the class of graphs attaining equality in such bound.

\begin{proposition}
{\em Let $G$ be a graph of order $n$. Then $\gamma^t_{StR}(G)=\gamma_t(G)$ if and only if $G$ is the disjoint union of copies of $K_2$.}
\end{proposition}

\begin{proof}
Assume $\gamma^t_{StR}(G)=\gamma_t(G)$, and let $f=(B_0,B_1,B_2)$ be a $\gamma^t_{StR}(G)$-function. Then $B_1 \cup B_2$ is a total dominating set of $G$. Thus, $$\gamma_t(G) \leq |B_1|+|B_2| \leq |B_1|+|B_2| +\sum_{{\color{black}w \in B_2}}\left\lceil \frac{1}{2}|N(w) \cap B_0|\right\rceil =\gamma^t_{StR}(G).$$
Then it must happen there is an equality situation in this inequality chain. In particular, it must happen $\sum_{w \in B_2}\frac{1}{2}\left\lceil {\color{black}|N(w) \cap B_0|} \right\rceil=0$, implying that $|B_2|=0$, and consequently, $V(G)=|B_1|$. Since $f$ is an arbitrary $\gamma^t_{StR}(G)$-function, $(\emptyset, V(G),\emptyset)$ is the only
$\gamma^t_{StR}(G)$-function. By Theorem \ref{th4}, $G$ is the disjoint union of copies of $K_2$, which completes this implication. The second implication is straightforward and the proof is complete.
\end{proof}

Based on the relatively simple deduction of the result above, we are next interested into those graphs $G$ for which $\gamma^t_{StR}(G) =\gamma_t(G)+1$.

\begin{proposition}\label{t+1}
{\em Let $G$ be a connected graph of order $n \geq 3$. Then, $\gamma^t_{StR}(G) =\gamma_t(G)+1$ if and only if
$G$ is $P_3$ or $C_3$.}
\end{proposition}

\begin{proof}
If $G$ is $P_3$ or $C_3$, then clearly $\gamma^t_{StR}(G) =\gamma_t(G)+1$.

Conversely, let $\gamma^t_{StR}(G) =\gamma_t(G)+1$ and let $f = (B_0,B_1,B_2)$ be a  $\gamma^t_{StR}(G)$-function.
If $B_2=\emptyset$, then $\gamma^t_{StR}(G)= n$, and by Theorem \ref{thn}, $G$ is one of the graphs in Theorem \ref{thn}, but clearly in such cases $\gamma_t(G)<n-1$ which is a contradiction.
Thus, $B_2 \neq \emptyset$.
Since $B_1 \cup B_2$ is a TD-set of $G$, we have
$$\gamma^t_{StR}(G)-1 = \gamma_t(G) \leq |B_1|+|B_2| \leq |B_1|+|B_2| +\sum_{w \in B_2}\left\lceil \frac{1}{2}|N(w) \cap B_0|\right\rceil-1=\gamma^t_{StR}(G)-1.$$
Consequently, this inequality chain must become a chain of equal quantities. In particular, $\sum_{w \in B_2}\left\lceil \frac{1}{2}|N(w) \cap B_0|\right\rceil=1$, which leads to $|B_2|=1$ and $\left\lceil\frac{1}{2}|N(w) \cap B_0|\right\rceil=1$ where $B_2=\{w\}$.
Therefore, $\gamma^t_{StR}(G)=\gamma_{tR}(G)$ and by Theorem \ref{th5}, $\Delta(G)=n-1$, which leads to $G=P_3$ or $G=C_3$.
\end{proof}

Another relationship between $\gamma^t_{StR}(G)$ and $\gamma_t(G)$ was already noticed in Observation \ref{1}. We are next interested into characterizing the limit case of such bound.

\begin{proposition}
{\em Let $G$ be a graph of order $n$ and $\Delta>1$. Then $\gamma_{StR}^t(G)=\left\lceil \frac{\Delta+1}{2}\right\rceil\gamma_t(G)$ if and only if  there exists a $\gamma^t_{StR}(G)$-function $f=(B_0,B_1,B_2)$ such that $|B_1|=0$ and $|N(w) \cap B_0|=\Delta-1$ for each $w \in B_2$.}
\end{proposition}

\begin{proof}
Let $\gamma^t_{StR}(G)=\left\lceil \frac{\Delta+1}{2}\right\rceil\gamma_t(G)$ and $S$ be an arbitrary $\gamma_t(G)$-set.
The function $f=(B_0,B_1,B_2)$ that
assigns the weight $1+ \left\lceil \frac{\Delta-1}{2}\right\rceil$ to each vertex of $S$, and the weight 0 to all remaining vertices
of $G$ is a TStRD-function on $G$. Thus,
\begin{align*}
\left\lceil \frac{\Delta+1}{2}\right\rceil \gamma_t(G)&=\gamma^t_{StR}(G)
\leq f(V(G))=|B_2| +\sum_{w \in B_2}\left\lceil \frac{1}{2}|N(w) \cap B_0|\right\rceil \\
&=\left(1+ \left\lceil \frac{\Delta-1}{2}\right\rceil\right)|S|= \left\lceil \frac{\Delta+1}{2}\right\rceil|S|\leq \left\lceil \frac{\Delta+1}{2}\right\rceil \gamma_t(G).
\end{align*}
We notice that the last inequality must be equality, since $S$ is a $\gamma_t(G)$-set. Also, observe that $|B_2|=|S|$. Thus, this inequality chain must become into a chain of equal quantities.
Particularly, $$\gamma^t_{StR}(G)= f(V(G))=\left\lceil \frac{\Delta+1}{2}\right\rceil|B_2| = |B_1|+|B_2| +\sum_{w \in B_2}\left\lceil \frac{1}{2}|N(w) \cap B_0|\right\rceil.$$
As a consequence, $|B_1|=0$ and $\left\lceil \frac{\Delta-1}{2}\right\rceil|B_2|=\sum_{w \in B_2}\left\lceil \frac{1}{2}|N(w) \cap B_0|\right\rceil$ hold,
which leads to that $f$ is a $\gamma_{StR}(G)$-function for which
$|B_1|=0$ and also $|N(w) \cap B_0|=\Delta-1$ for each $w \in B_2$.

Conversely, assume there exists a $\gamma^t_{StR}(G)$-function $f=(B_0,B_1,B_2)$ such that $|B_1|=0$ and that $|N(w) \cap B_0|=\Delta-1$ for each $w \in B_2$.
Since $B_1 \cup B_2=B_2$ is a total dominating set of $G$, we have
$$\gamma_t(G) \leq |B_2| \leq \frac{1}{1+ \left\lceil \frac{\Delta-1}{2}\right\rceil}\gamma^t_{StR}(G)=\frac{1}{\left\lceil \frac{\Delta+1}{2} \right\rceil}\gamma^t_{StR}(G).$$
Thus, it follows $\gamma_{StR}^t(G)\geq \left\lceil \frac{\Delta+1}{2}\right\rceil \gamma_t(G)$. Therefore, by using Observation \ref{1}, we obtain the equality $\gamma^t_{StR}(G)=\left\lceil \frac{\Delta+1}{2}\right\rceil\gamma_t(G)$.
\end{proof}

Our final results in this section relate the total strong Roman domination number and the (standard) domination number of graphs.

\begin{theorem}\label{three}
{\em Let $G$ be a graph without isolated vertices. Then,
$$\gamma^t_{StR}(G) \le \left(\left\lceil  \frac{\Delta-1}{2}\right\rceil+2\right)\gamma(G).$$
Moreover, if the equality holds, then every $\gamma(G)$-set $S$ is an efficient dominating set, and every vertex in $S$ has degree $\Delta$.}
\end{theorem}

\begin{proof}  Let $S$ be a $\gamma(G)$-set, and consider $S'$ represents the set of vertices in $S$
that have degree zero in $G[S]$ (it could happen, $S' = \emptyset$). For every vertex $v \in S'$ (if it exists), we chose one vertex adjacent to $v$, and denote such vertex as $v'$.
Let $S'' = \cup_{v \in S'} \{v'\}$. Let $f$ be  a TStRD-function on $G$ defined as follows:
(a) for each vertex $v \in S$, let $f(v) =\left\lceil  \frac{\Delta-1}{2}\right\rceil+1$,
(b) for each vertex $v \in S'' $, let $f(v) = 1$, and
(c) for each vertex $v \in V(G) \setminus (S \cup S'')$, let $f(v) = 0$.
Thus, it follows $\gamma^t_{StR}(G) \le f(V(G)) \le (\left\lceil  \frac{\Delta-1}{2}\right\rceil+1)|S|+|S''| \le (\left\lceil  \frac{\Delta-1}{2}\right\rceil+2)|S| = (\left\lceil  \frac{\Delta-1}{2}\right\rceil+2) \gamma(G)$, as desired.

Assume next $\gamma^t_{StR}(G) = (\left\lceil  \frac{\Delta-1}{2}\right\rceil+2)\gamma(G)$.
Let $S$ be any $\gamma(G)$-set, and let $S'$ and $S''$ be two sets defined as previously described.
It is true that $(\left\lceil  \frac{\Delta-1}{2}\right\rceil+2)\gamma(G) = \gamma^t_{StR}(G) \le (\left\lceil  \frac{\Delta-1}{2}\right\rceil+1)|S| + |S''| \le (\left\lceil  \frac{\Delta-1}{2}\right\rceil+1)|S| + |S'| \le (\left\lceil  \frac{\Delta-1}{2}\right\rceil+2)|S| = (\left\lceil  \frac{\Delta-1}{2}\right\rceil+2)\gamma(G)$.
As a consequence, an equality relation must occur in this last inequality chain.
So, $|S''| = |S'| = |S|$,  which leads to claim that $S$ is an independent set. Moreover,
every vertex belonging to the set $S$ is of degree $\Delta$.
But then  $d_G(u,v) \ge 2$ (distance between $u$ and $v$) for any two distinct vertices $u$ and $v$  in $S$.
 We shall show that $d_G(u,v) \ge 3$. For a contradiction purpose, we suppose $d_G(u,v) = 2$.
Let $w$ be a vertex adjacent to both vertices $u$ and $v$, and select $u' = v' = w$ where,
as above, the vertices $u'$ and $v'$ are those vertices chosen to be neighbors of $u$ and $v$, respectively.
According to this, we note that $|S''| < |S'|$, which produces a contradiction.
 Therefore, we must have $d_G(u,v) \ge 3$ for any pair of distinct vertices $u$ and $v$ from the set $S$.
Consequently, the set $S$ is an efficient dominating set in the graph $G$.
Therefore, we have obtained that every $\gamma(G)$-set $S$ is an efficient dominating set in $G$, and that every vertex in $S$ has degree $\Delta$.
\end{proof}

\begin{proposition}
{\em Let $G$ be a connected graph of order $n$ and without isolated vertices. Then $\gamma(G)+\gamma^t_{StR}(G) \leq \frac{3n}{2}$.
The equality holds if and only if $G=C_4$ or $G$ is a corona, $cor(F )$, of some graph $F$.}
\end{proposition}
\begin{proof}
By Theorem \ref{th6}(a), $\gamma (G) \leq \frac{n}{2}$,
and by Observation \ref{o2}, $ \gamma^t_{StR}(G) \leq n$.
Hence, $\gamma(G) + \gamma^t_{StR}(G) \leq \frac{3n}{2}$.
The equality holds if and only if both $\gamma (G) = \frac{n}{2}$ and
$ \gamma^t_{StR}(G) = n$ are valid. The required result follows by combining Theorem \ref{th6}(b)
and Theorem \ref{thn}.
\end{proof}
%%%%%%%%%%%%%%%%%%%%%%%%%

%%%%%%%%%%%%%%%%%%%%%%%%%%%%%%%%%%%%%%%%%%%%%%%%%%%%%%%%%%%%%%%%%%%%%%%%%%%%%%%%%%%%%%%%%%%%%
\section{Trees}

In this section, we present two bounds for the total strong Roman domination number of trees. First, we note that by Proposition \ref{t+1}, for any tree $T$ of order $n\ge 4$, it follows $\gamma_{StR}^{t}(T)\ge  \gamma_t(T)+2$. We shall improve such bound whether we have maximum degree larger than five. To this end, we need the following lemma.

\begin{lemma}\label{leaves-zero}
{\em Let $T$ be a tree different from a star. Then, there exists a $\gamma^t_{StR}(T)$-function $f=(B_0, B_1, B_2)$ such that {\color{black} every leaf of $T$ belongs to $B_0$.}
}
\end{lemma}

\begin{proof}
Let $f=(B_0, B_1, B_2)$ be a $\gamma^t_{StR}(T)$-function. Suppose there exist a leaf $x$ of $T$ such that $f(x)\ge 1$ and let $x'$ be the support vertex adjacent to $x$. {\color{black}By definition $f(x')\ge 1$. If $f(x)\ge 2$, then} one can easily construct a TSrRD-function with weight smaller than $f$, which is not possible. {\color{black}Thus, it must happen $f(x)=1$.} If the support vertex $x'$ has $t\ge 2$ adjacent leaves labeled with one, then we can construct a new TStRD-function of $T$ of weight smaller than or equal to $\gamma^t_{StR}(T)$, by relabeling $t-1$ of such leaves with zero, and the vertex $x'$ with $f(x')+\left\lceil (t-1)/2\right\rceil$. In consequence, we may consider that every support vertex has at most one leaf labeled with one under $f$, say our $x$, for the support vertex $x'$. If there exists a not leaf vertex $y\in N(x')\setminus\{x\}$ for which $f(y)=0$, then one can ``exchange'' the labels of $x$ and $y$ to construct a new $\gamma^t_{StR}(T)$-function satisfying $f(x)=0$. Hence, we may assume $f(y)\ne 0$ for every not leaf $y\in N(x')\setminus\{x\}$. This also leads to conclude that $f(x')=1+\left\lceil \frac{\vert L_{x'}\vert-1}{2} \right\rceil$, otherwise we can decrease the weight of $f$, which is not possible.
But, then we can construct a new $\gamma^t_{StR}(T)$-function by changing the labels of $x$ and $x'$ to zero and $1+\left\lceil \frac{\vert L_{x'}\vert}{2} \right\rceil$, respectively, and this is either not possible or satisfies our requirement.
\end{proof}

\begin{theorem} \label{stR-t}
{\em For any nontrivial tree $T$ with maximum degree $\Delta(T)$,
$$\gamma_{StR}^{t}(T)\ge  \gamma_t(T)+\left\lceil\frac{\Delta(T)-1}{2}\right\rceil.$$
This bound is sharp for stars.}
\end{theorem}

\begin{proof}
By Proposition \ref{t+1}, for any tree $T$ of order $n\ge 4$, it follows $\gamma_{StR}^{t}(T)\ge  \gamma_t(T)+2$. Thus, if $\Delta(T)\le 5$, then we deduce that $\gamma_{StR}^{t}(T)\ge \gamma_t(T)+2 \ge \gamma_t(T)+\left\lceil\frac{\Delta(T)-1}{2}\right\rceil.$ In consequence, from now on we may assume $\Delta(T)\ge 6$. This implies that $n\ge 7$, and also that $\rm{diam}(T)\ge 2$.

If $\rm{diam}(T)=2$, then $T$ is a star with $\left\lceil\frac{\Delta(T)-1}{2}\right\rceil+2=\gamma_{StR}^{t}(T)=\left\lceil\frac{\Delta(T)-1}{2}\right\rceil+\gamma_t(T)$ (this also shows the sharpness of the bound). If $\rm{diam}(T)=3$, then $T$ is a double star $DS_{p,q}, (p\ge 1, q\ge 5)$, where $\Delta(T)=q+1$, $\gamma_t(T)=2$ and
$\gamma_{StR}^{t}(T)=\left\lceil\frac{p}{2}\right\rceil+1+\left\lceil\frac{q}{2}\right\rceil+1$. By the fact that $\left\lceil\frac{p}{2}\right\rceil\ge 1$, we have $\left\lceil\frac{p}{2}\right\rceil+\left\lceil\frac{q}{2}\right\rceil+2> \left\lceil\frac{q}{2}\right\rceil+2$. Hence $\gamma_{StR}^{t}(T)>  \gamma_t(T)+\left\lceil\frac{\Delta(T)-1}{2}\right\rceil$. The remaining part of the proof shall be done by induction on the order $n$ of $T$.

Let $T$ be a tree of order $n$ and $\rm{diam}(T)\ge 4$, and assume that any tree $T'$ of order $n'<n$ and $\Delta(T')\ge 6$ satisfies $\gamma_{StR}^{t}(T')\ge \gamma_t(T')+\left\lceil\frac{\Delta(T')-1}{2}\right\rceil$.
Suppose ${\rm diam}(T) = k-1$, and let $P := v_1,v_2,\ldots, v_{k}$ be a diametrical path of $T$ such that $v_2$ has the smallest possible degree. Root $T$ at $v_k$ and let $T'=T-\{v_1\}$.

Note that by the choice of $P$, the removing of the vertex $v_1$ does not change the maximum degree of $T'$ with respect to that of $T$. Thus $\Delta(T')=\Delta(T)\ge 6$.
On the other hand, it clearly happens that $\gamma_t(T)-1\le \gamma_t(T')\le\gamma_t(T)$. Also, by Lemma \ref{leaves-zero}, there exists a $\gamma^t_{StR}(T)$-function $f=(B_0, B_1, B_2)$ such that for every leaf $x$ of $T$, it follows $f(x)=0$. This allows to claim that
$\gamma^t_{StR}(T)\ge \gamma^t_{StR}(T')$. Now, if $\gamma_t(T')=\gamma_t(T)$, then by using the induction hypothesis, we deduce
\begin{align*}
\gamma_{StR}^{t}(T) &\ge \gamma_{StR}^{t}(T')\hfill\\
&\ge \gamma_t(T')+\left\lceil\frac{\Delta(T')-1}{2}\right\rceil\hfill\\
&= \gamma_t(T)+\left\lceil\frac{\Delta(T)-1}{2}\right\rceil\hfill.
\end{align*}
We next consider $\gamma_t(T')=\gamma_t(T)-1$. We then observe that $v_2$ must have degree two in $T$, since otherwise the removing of $v_1$ to obtain $T'$ will not do $\gamma_t(T')$ strictly smaller than $\gamma_t(T)$. Since $f(v_1)=0$, it must happen that $f(v_2)\ge 2$. Moreover, the neighbor of $v_2$ (which is indeed $v_3\in P$) other than $v_1$ satisfies $f(v_3)\ge 1$. It is clearly now deduced that $f(v_2)=2$, for otherwise we can construct a TStRD-function of $T$ of weight smaller than that of $f$, which is not possible. Thus, from $f$, we construct a TStRD-function $f'$ in $T'$ by taking the restriction of $f$ to $T'$ and only changing the label of $v_2$, that is, making $f'(v_2)=1$. So, we obtain that $\gamma^t_{StR}(T')\le \gamma^t_{StR}(T)-1$. Therefore, by using the induction hypothesis, we deduce
\begin{align*}
\gamma_{StR}^{t}(T) &\ge \gamma_{StR}^{t}(T')+1\hfill\\
&\ge \gamma_t(T')+\left\lceil\frac{\Delta(T')-1}{2}\right\rceil+1\hfill\\
&= \gamma_t(T)-1+\left\lceil\frac{\Delta(T)-1}{2}\right\rceil+1\hfill\\
&= \gamma_t(T)+\left\lceil\frac{\Delta(T)-1}{2}\right\rceil,
\end{align*}
which completes the proof.
\end{proof}

%%%%%%%%%%%%%%%%%%%%%%%%%%%
%%%%%%%%%%%%%%%%%%%%%%%%%%%%%%%%%%%%%%%%%%%%%%%%%%%%%%%%%%%%%%%
\begin{theorem}
{\em For any tree $T$ of order $n(T)\ge 3$ with maximum degree $\Delta(T)$ and $s(T)$ support vertices,
$$\gamma_{stR}^t(T)\ge \left\lceil \frac{n(T)+s(T)}{\Delta(T)}\right\rceil+1.$$
Furthermore, this bound is sharp for $T\in \{P_3,P_4,P_5,DS_{1,2},DS_{2,2},K_{1,3}\}$.}
\end{theorem}

\begin{proof}
The proof shall be made by induction on the order $n(T)$ of the tree $T$. One can verify that the statement is true whether the order $n(T)\le 3$. Hence, we now on in this proof consider $n(T)\geq 4$, and that every tree $T'$ of order $n(T')<n(T)$ with $s(T')$ support vertices satisfies the bound $\gamma_{stR}^t(T')\ge \left\lceil \frac{n(T')+s(T')}{\Delta(T')}\right\rceil+1$. If $T=K_{1,n-1}$ is a star, then $\gamma_{stR}^t(T)=\left\lceil\frac{n+2}{2}\right\rceil\ge\left\lceil\frac{n+1}{n-1}\right\rceil+1$ and the equality holds for $n=4$.
Likewise, assume that $T$ is a double star. If $p=q=1$, then $T=P_4$ and $\gamma_{stR}^t(T)=4=\left\lceil\frac{6}{2}\right\rceil+1$.
If $T=DS_{1,q}$, with $q\ge 2$, then we have $\gamma_{stR}^t(T)=\left\lceil\frac{q}{2}\right\rceil+3\ge\left\lceil\frac{q+5}{q+1}\right\rceil+1$ and the equality holds for $q=2$.
If $T=DS_{p,q}$, with $q\ge p\ge 2$, then $\gamma_{stR}^t(T)=\left\lceil\frac{q}{2}\right\rceil+\left\lceil\frac{p}{2}\right\rceil+2\ge \left\lceil\frac{p+5}{q+1}\right\rceil+1$ and the equality holds for $p=q=2$. If $T$ is a path $P_n$, then by using Observation \ref{o3} and Theorem \ref{path2}, we have $\gamma_{stR}^t(T)=n\ge \left\lceil \frac{n+2}{2}\right\rceil+1$, and the equality holds if $n=4$ or $n=5$.

Consequently, we may assume that ${\rm diam}(T)\ge 4$ and that $\Delta(T)\ge 3$.
Let $v_{1}v_{2}\ldots v_{k}$ be a diametrical path in $T$ such that $v_2$ has the smallest possible degree. Note that all the descendants of $v_2$ are leaves adjacent to $v_2$. We consider the tree $T$ rooted at the vertex $v_{k}$, and analyze the following situations.

\smallskip
\noindent
\textbf{Case 1:} $v_2$ has degree two. Let $T'=T-\{v_1\}$. As in the proof of Theorem \ref{stR-t}, by the choice of $P$, the removing of the vertex $v_1$ does not change the maximum degree of $T'$ with respect to that of $T$. Thus $\Delta(T')=\Delta(T)\ge 3$. Also, $v_2$ is not a support vertex of $T'$, and by using the same idea as in the proof of Theorem \ref{stR-t}, we deduce that $\gamma_{StR}^{t}(T)\ge \gamma_{StR}^{t}(T')+1$. Thus, by using the induction hypothesis and taking into account that $s(T)=s(T')+1$ and $n(T)=n(T')+1$, it follows,
\begin{align*}
\gamma_{StR}^{t}(T) &\ge \gamma_{StR}^{t}(T')+1\hfill\\
&\ge \left\lceil\frac{n(T')+s(T')}{\Delta(T')}\right\rceil+2\hfill\\
&\ge \left\lceil\frac{n(T)-1+s(T)-1}{\Delta(T)}\right\rceil+2\hfill\\
&= \left\lceil\frac{n(T)+s(T)+\Delta(T)-2}{\Delta(T)}\right\rceil+1\hfill\\
&\ge \left\lceil\frac{n(T)+s(T)}{\Delta(T)}\right\rceil+1.
\end{align*}

\noindent
\textbf{Case 2:} $v_2$ has degree larger than three. This means $v_2$ has at least three adjacent leaves. Let $v'_1$ be a leaf adjacent to $v_2$ other than $v_1$ and let $T''=T-\{v_1,v'_1\}$. Again, the removing of the vertices $v_1,v'_1$ does not change the maximum degree, and so $\Delta(T'')=\Delta(T)\ge 3$. Moreover, by Lemma \ref{leaves-zero}, there exists a $\gamma^t_{StR}(T)$-function $f=(B_0, B_1, B_2)$ such that for every leaf $x$ of $T$, it follows $f(x)=0$. In this sense, since $f(v_1)=f(v'_1)=0$, we can construct a TStRD-function for $T''$ by decreasing the label of $v_2$ by one, and maintaining the remaining labels unchanged. This leads to claim $\gamma^t_{StR}(T)\ge \gamma^t_{StR}(T'')+1$. Thus, since $v_2$ continues being a support vertex in $T''$, by using the induction hypothesis and the equalities $s(T)=s(T'')$ and $n(T)=n(T'')+2$, we deduce,
\begin{align*}
\gamma_{StR}^{t}(T) &\ge \gamma_{StR}^{t}(T'')+1\hfill\\
&\ge \left\lceil\frac{n(T'')+s(T'')}{\Delta(T'')}\right\rceil+2\hfill\\
&\ge \left\lceil\frac{n(T)-2+s(T)}{\Delta(T)}\right\rceil+2\hfill\\
&= \left\lceil\frac{n(T)+s(T)+\Delta(T)-2}{\Delta(T)}\right\rceil+1\hfill\\
&\ge \left\lceil\frac{n(T)+s(T)}{\Delta(T)}\right\rceil+1.
\end{align*}

\noindent
\textbf{Case 3:} $v_2$ has degree three. In order to simplify the proof, we shall adapt Case 2 to this situation. We again define $T''=T-\{v_1,v'_1\}$. Clearly, now $v_2$ is not a support vertex, but a leaf in $T''$. Since $v_2$ has degree three in $T'$ and is adjacent to two leaves $v_1,v'_1$ such that $f(v_1)=f(v'_1)=0$, it must happen that $f(v_3)\ge 1$. Consequently, it must happen $f(v_2)=2$. Thus, we construct a new TStRD-function on $T''$ by relabeling $v_2$ with one, which means $\gamma^t_{StR}(T)\ge \gamma^t_{StR}(T'')+1$. Now, since $v_2$ is not a support vertex in $T''$, by using the induction hypothesis and the equalities $s(T)=s(T'')+1$ and $n(T)=n(T'')+2$, we deduce that (the last inequality follows since $\Delta(T)\ge 3$),
\begin{align*}
\gamma_{StR}^{t}(T) &\ge \gamma_{StR}^{t}(T'')+1\hfill\\
&\ge \left\lceil\frac{n(T'')+s(T'')}{\Delta(T'')}\right\rceil+2\hfill\\
&\ge \left\lceil\frac{n(T)-2+s(T)-1}{\Delta(T)}\right\rceil+2\hfill\\
&= \left\lceil\frac{n(T)+s(T)+\Delta(T)-3}{\Delta(T)}\right\rceil+1\hfill\\
&\ge \left\lceil\frac{n(T)+s(T)}{\Delta(T)}\right\rceil+1,
\end{align*}
and this completes the proof.
\end{proof}

%###########################################################################################################################
%\begin{thebibliography}{99}


\begin{thebibliography}{99}
\bibitem{ah} H. Abdollahzadeh Ahangar, M.A. Henning, V. Samodivkin and I.G.
Yero, \emph{Total Roman domination in graphs}, Appl. Anal. Discrete Math.
\textbf{10} (2016) 501--517.

\bibitem{ah1}
H. Abdollahzadeh Ahangar, J. Amjadi, M. Chellali, S. Nazari-Moghaddam and
S. M. Sheikholeslami, {\em Total Roman reinforcement in graphs}, Discuss. Math. Graph Theory {\color{black}{\bf 39} (2019), 787--803.}

\bibitem{ah2}
H. Abdollahzadeh Ahangar, J. Amjadi, S. M. Sheikholeslami, M. Soroudi, {\em Bounds on the total Roman domination number of graphs}, Ars Combin. (to appear).

\bibitem{ah2-3}H. Abdollahzadeh Ahangar, J. Amjadi, S. M. Sheikholeslami, M. Soroudi, {\em Total Roman domination and 2-independence in trees}, Ars Combin. (to appear).

\bibitem{S1} M. P. \'{A}lvarez-Ruiz, T. Mediavilla-Gradolph,
S. M. Sheikholeslami, I. G. Yero, and J.C. Valenzuela-Tripodoro, {\em On the
strong Roman domination number of graphs}, {\color{black}Discrete App.
Math.} {\bf 231 } (2017) 44--59.

\bibitem{ah3}
J. Amjadi, S. Nazari-Moghaddam and S. M. Sheikholeslami, {\em Global total Roman domination in graphs}, Discr. Math. Algorithms Appl. {\bf 9} (2017), 1750050, 13 pp.

\bibitem{am} J. Amjadi, S. Nazari-Moghaddam and S.M. Sheikholeslami, Total Roman domatic number of a graph, Asian-European J. Math. (to appear)

\bibitem{ah4}
 J. Amjadi, S. Nazari-Moghaddam, S. M. Sheikholeslami and L. Volkmann, {\em Total
Roman domination number of trees}, Austr. J. Combin. {\bf 69} (2017) 271--285.

\bibitem{ah54}J. Amjadi, S. M. Sheikholeslami and M. Soroudi, {\em Nordhaus-Gaddum bounds for total
Roman domination}, J. Comb. Optim. {\bf 35} (2018) 126--133.

\bibitem{ASS}
J. Amjadi, S. M. Sheikholeslami and M. Soroudi, {\em On the total Roman domination in
trees}, Discuss. Math. Graph Theory {\bf 39} (2019) 519--532.

\bibitem{ah5} J. Amjadi and M. Soroudi, {\em Twin signed total Roman domination numbers in digraphs}, Asian-Eur. J. Math. {\bf 11} (2018) 22 pages.

\bibitem{ck}  N. Campanelli and D. Kuziak, {\em Total Roman domination in the lexicographic product of graphs}, Discrete Appl. Math. {\bf 263} (2019), 88--95.

\bibitem{C}
E. W. Chambers, B. Kinnersley, N. Prince and D. B. West, {\em Extremal
problems for Roman domination}, SIAM J. Discrete Math. {\bf
23}, (2009), 1575--1586.

\bibitem {cjsv}M. Chellali, N. Jafari Rad, S.M. Sheikholeslami and L.
Volkmann, Roman domination in graphs, In: \emph{Domination in Graphs: Major
Parameters}, Eds. T.W. Haynes, S.T. Hedetniemi and M.A. Henning, to appear 2020.

\bibitem {cjsv1}M. Chellali, N. Jafari Rad, S.M. Sheikholeslami and L.
Volkmann, Varieties of Roman domination, In: \emph{Domination in Graphs: Major
Topics}, Eds. T.W. Haynes, S.T. Hedetniemi and M.A. Henning, to appear 2020.

\bibitem {cjsv2}M. Chellali, N. Jafari Rad, S.M. Sheikholeslami and L.
Volkmann, \emph{Varieties of Roman domination II}, \textbf{(Submitted)}.

\bibitem {cjsv3}M. Chellali, N. Jafari Rad, S.M. Sheikholeslami and L.
Volkmann, \emph{A survey on Roman domination parameters in directed graphs},
\textbf{(Submitted)}.

\bibitem {cjsv4}M. Chellali, N. Jafari Rad, S.M. Sheikholeslami and L.
Volkmann, \emph{The Roman domatic problem in graphs and digraphs: A survey},
\textbf{(Submitted)}.


\bibitem{cdh} E.J. Cockayne, R.M. Dawes and S.T. Hedetniemi, {\em Total
domination in graphs}, Networks \textbf{10} (1980), 211--219.

\bibitem{CDH04} E. J. Cockayne, P. A. Dreyer, S M. Hedetniemi and S. T.
Hedetniemi, {\em Roman domination in graphs}, Discrete Math. \textbf{278}
(1-3) (2004), 11--22.

%\bibitem{FaKaKhSh09} O. Favaron, H. Karami, R. Khoeilar, and S.M.
%Sheikholeslami, {\em On the Roman domination number of a graph}, Discrete
%Math. \textbf{309} (2009), 3447--3451.


\bibitem{fjkr}  J. F. Fink, M. S. Jacobson, L. F. Kinch,  J. Roberts, {\em On graphs having
domination number half their order}, Period. Math. Hungar. (1985), 287--293.

\bibitem{GJ} M.R. Garey and D.S. Johnson, {\em Computers and Intractability: A
Guide to the Theory of NP-Completeness}, Freeman, San Francisco, (1979).

\bibitem{GH} W. Goddard and M.A. Henning, \emph{Independent domination in graphs: A survey and recent results}, Discrete Math. \textbf{313} (7) (2013) 839--854.

\bibitem{hhs} T.W. Haynes, S.T. Hedetniemi and P. J. Slater, {\em
Fundamentals of Domination in graphs}. Marcel Dekker, Inc., New york (1998).

\bibitem{hhs-2} T.W. Haynes, S.T. Hedetniemi and P. J. Slater, {\em
Fundamentals of Domination in graphs: Advanced Topics}. Chapman \& Hall, CRC Press (1998).

\bibitem{He09} M.A. Henning, {\em Recent results on total domination in
graphs: A survey}, Discrete Math. \textbf{309} (2009),
32--63.

\bibitem{HeYe_book} M. A. Henning and A. Yeo, {\em Total domination
in graphs}, Springer Monographs in Mathematics (2013), ISBN:
978-1-4614-6524-9 (Print) 978-1-4614-6525-6 (Online).

%\bibitem{LiCh12a} C.-H. Liu and G. J. Chang, {\em Roman domination on $2$
%-connected graphs}, SIAM J. Discrete Math. \textbf{26}
%(2012), 193--205.
%
%\bibitem{LiCh12b} C.-H. Liu and G. J. Chang, {\em Upper bounds on Roman
%domination numbers of graphs}, Discrete Math. \textbf{312}
%(2012), 1386--1391.

\bibitem{LiCh13} C.-H. Liu and G. J. Chang, {\em Roman domination on
strongly chordal graphs}, J. Comb. Optim. \textbf{26} (2013),
608--619.

\bibitem{mky} A.C. Mart\'inez, D. Kuziak and I.G. Yero, {\em Outer-independent total Roman domination in graphs}, Discrete Appl. Math. {\bf 269} (2019), 107--119.

\bibitem{ore} O. Ore, {\em Theory of Graphs}, Amer. Math. Soc. Colloq. Publ, 38 (Amer.
Math. Soc, Providence, RI), (1962).

%\bibitem{PaZe12} P. Pavli\v{c} and J. \v{Z}erovnik, {\em Roman domination number
%of the Cartesian products of paths and cycles}, Electron. J.
%Combin. \textbf{19}(3) (2012), 1-- 37.

\bibitem{px}   C. Payan, N. H. Xuong, {\em Domination-balanced graphs}, J. Graph Theory {\bf 6} (1982), 23--32.

%\bibitem{rr} C. S. Revelle and K. E. Rosing, {\em Defendens imperium
%romanum: a classical problem in military strategy}, Amer. Math.
%Monthly \textbf{107}(7) (2000), 585--594.

\bibitem{s} I. Stewart, {\em Defend the Roman Empire}, Sci. Amer. \textbf{28}(6) (1999), 136--139.


\bibitem{GoRoJu13} I. G. Yero and J. A. Rodr\'{\i}guez-Vel\'{a}zquez, {\em Roman
domination in Cartesian product graphs and strong product graphs},
Appl. Anal. Discrete Math. \textbf{7} (2013), 262--274.
\end{thebibliography}
\end{document}